\documentclass[12pt]{amsart}
\usepackage{amscd, amsfonts, amssymb, mathrsfs}
\usepackage{fullpage}
\usepackage{latexsym}
\usepackage{verbatim}
\usepackage{enumerate}
\usepackage{enumitem}
\usepackage[all,cmtip]{xy}
\usepackage[colorlinks=true]{hyperref}
\numberwithin{equation}{section}

\newtheorem{thm}[equation]{Theorem}

\newtheorem{lem}[equation]{Lemma}
\newtheorem{cor}[equation]{Corollary}
\newtheorem{prop}[equation]{Proposition}
\newtheorem{qn}[equation]{Question}

\theoremstyle{definition}
\newtheorem{defn}[equation]{Definition}
\newtheorem{ex}[equation]{Example}
\newtheorem{exmp}[equation]{Example}

\theoremstyle{remark}
\newtheorem{rem}[equation]{Remark}
\theoremstyle{remark}
\newtheorem{rems}[equation]{Remarks}

\DeclareMathOperator{\Aut}{Aut}
\DeclareMathOperator{\GL}{GL}
\DeclareMathOperator{\End}{End}
\DeclareMathOperator{\Hom}{Hom}
\DeclareMathOperator{\SL}{SL}
\DeclareMathOperator{\Sp}{Sp}
\DeclareMathOperator{\So}{SO}
\DeclareMathOperator{\SU}{SU}
\DeclareMathOperator{\PGL}{PGL}

\DeclareMathOperator{\Lie}{Lie}

\DeclareMathOperator{\Ad}{Ad}
\DeclareMathOperator{\rk}{rk}
\DeclareMathOperator{\ad}{ad}

\newcommand{\sat}{{\rm sat}}

\newcommand{\BBF}{\mathbb{F}}
\newcommand{\BBN}{\mathbb{N}}
\newcommand{\BBG}{\mathbb{G}}
\newcommand{\NN}{\mathcal{N}}
\newcommand{\UU}{\mathcal{U}}

\renewcommand{\gg}{\mathfrak g}
\newcommand{\cc}{\mathfrak c}
\newcommand{\zz}{\mathfrak z}
\newcommand{\mm}{\mathfrak m}
\newcommand{\gl}{\mathfrak{gl}}
\renewcommand{\sl}{\mathfrak{sl}}

\newcommand{\lima}{\lim\limits_{a\to 0}}
\newcommand{\limau}{\lim\limits_{a\to 0} \lambda(a)u\lambda(a)^{-1}}
\newcommand{\limaut}{\lim\limits_{a\to 0} \lambda(a)u^t\lambda(a)^{-1}}

\newcommand{\twobytwo}[4]{{\left(\begin{array}{ll} #1 & #2 \\ #3 & #4 \\\end{array}\right)}}

\subjclass[2010]{20G15 (14L24)}
\keywords{$G$-complete reducibility; saturation; finite groups of Lie type}

\title[$G$-complete reducibility and saturation]
{$G$-complete reducibility and saturation}

\author[M.\  Bate]{Michael Bate}
\address%[M.\  Bate]
{Department of Mathematics,
University of York,
York YO10 5DD,
United Kingdom}
\email{michael.bate@york.ac.uk}

\author[S. B\"ohm]{S\"oren B\"ohm}
\address%[G.~R\"{o}hrle]
{Fakult\"at f\"ur Mathematik,
	Ruhr-Universit\"at Bochum,
	D-44780 Bochum, Germany}
\email{soeren.boehm@rub.de}

\author[A.\ Litterick]{Alastair Litterick}
\address{School of Mathematics, Statistics and Actuarial Science, University of Essex, Wivenhoe Park, Colchester, Essex CO4 3SQ, United Kingdom}
\email{a.litterick@essex.ac.uk}

\author[B.\ Martin]{Benjamin Martin}
\address%[B.\ Martin]
{Department of Mathematics,
University of Aberdeen,
King's College,
Fraser Noble Building,
Aberdeen AB24 3UE,
United Kingdom}
\email{b.martin@abdn.ac.uk}
 
\author[G. R\"ohrle]{Gerhard R\"ohrle}
\address%[G.~R\"{o}hrle]
{Fakult\"at f\"ur Mathematik,
Ruhr-Universit\"at Bochum,
D-44780 Bochum, Germany}
\email{gerhard.roehrle@rub.de}

\begin{document}

\begin{abstract}
	Let $H \subseteq G$ be connected reductive linear algebraic  groups defined over an algebraically closed field of characteristic $p> 0$.
	In our first main theorem we show that if a closed subgroup $K$ of $H$ is $H$-completely reducible, then it is also $G$-completely reducible in the sense of Serre, under some restrictions on $p$, generalising the known case for $G = \GL(V)$.  Our proof uses R.W. Richardson's notion of reductive pairs to reduce to the $\GL(V)$ case.
	We study Serre's notion of saturation and prove that saturation behaves well with respect to products and regular subgroups.  Our second main theorem shows that if $K$ is $H$-completely reducible, then the saturation of $K$ in $G$ is completely reducible in the saturation of $H$ in $G$ (which is again a connected reductive subgroup of $G$), under suitable restrictions on $p$, again generalising the known instance for $G = \GL(V)$.
	We also study saturation of finite subgroups of Lie type in $G$.  We show that saturation is compatible with standard Frobenius endomorphisms, and we use this to generalise a result due to Nori from  1987 in case $G = \GL(V)$.
\end{abstract}

\maketitle

\section{Introduction and main results}
\label{sec:intro}

Let $G$ be a connected reductive linear algebraic group over an algebraically closed field $k$ of characteristic $p> 0$.  Let $H$ be a closed subgroup of $G$.
Following  Serre \cite{serre2}, we say that 
$H$ is \emph{$G$-completely reducible} ($G$-cr for short) 
provided that whenever $H$ is contained in a parabolic subgroup $P$ of $G$,
it is contained in a Levi subgroup of $P$. Further,  $H$ is \emph{$G$-irreducible} ($G$-ir for short) provided $H$ is not contained in any proper parabolic subgroup of $G$ at all. 
Clearly, if $H$ is $G$-irreducible, it is trivially $G$-completely reducible; 
for an overview of this concept see \cite{BMR}, \cite{serre1} and \cite{serre2}.
Note in case $G = \GL(V)$ a subgroup $H$ is $G$-cr exactly when 
$V$ is a semisimple $H$-module and it is $G$-ir precisely when $V$ is an irreducible $H$-module. The same equivalence applies to $G = \SL(V)$.   The notion of $G$-complete reducibility is a powerful tool for investigating the subgroup structure of $G$: see \cite{LT}, \cite{LST}.

Now suppose $H$ is connected and reductive, and let $K$ be a closed subgroup of $H$.  It is natural to ask the following questions:

\begin{qn}
\label{qn:ascent}
 If $K$ is $H$-completely reducible, must $K$ be $G$-completely reducible?
\end{qn}

\begin{qn}
\label{qn:descent}
 If $K$ is $G$-completely reducible, must $K$ be $H$-completely reducible?
\end{qn}

\noindent The answer to Question~\ref{qn:ascent} is no in general.  For instance, if $H$ is a non-$G$-cr subgroup of $G$ and $K= H$ then $K$ is $H$-ir but $K$ is not $G$-cr.  For a more complicated example showing that the answer to both Questions~\ref{qn:ascent} and \ref{qn:descent} is no, see \cite[Ex.\ 3.45]{BMR}, or \cite[Prop.~7.17]{BMRT}.   On the other hand, if $p> 2$ and $H$ is either the special orthogonal group $\So(V)$ or the symplectic $\Sp(V)$ with its natural embedding in $G:= \GL(V)$ then $K$ is $G$-cr if and only if $K$ is $H$-cr (see \cite[Ex.\ 3.23]{BMR}).

In this paper we consider some variations on Questions~\ref{qn:ascent} and \ref{qn:descent}.  To do this we use two key tools: reductive pairs and saturation.  Our first main result shows that the answer to Question~\ref{qn:ascent} is yes if we impose some extra conditions on $p$.  To state our result we need some notation.  We define an invariant $d(G)$ of $G$ as follows.
For $G$ simple let 
$d(G)$ be as follows:
	\begin{center}
	\begin{tabular}{c|*{9}{c}}
		$G$ & $A_n \ (n\ge1)$ & $B_n \ (n\ge3)$ & $C_n \  (n\ge2)$ & $D_n \  (n\ge4)$ & $E_6$ & $E_7$ & $E_8$ & $F_4$ & $G_2$ \\ \hline
		$d(G)$ & $n+1$ & $2n+1$ & $2n$ & $2n$ & $27$ & $56$ & $248$ & $26$ & $7$
	\end{tabular}
\end{center}
For $G$ reductive, let $d(G) = \max(1, d(G_1), \ldots, d(G_r))$, 
where  $G_1, \ldots, G_r$ are the simple components of $G$.
For $G$ simple and simply-connected and $p$ good for $G$, $d(G)$ is the minimal possible dimension of a non-trivial irreducible $G$-module. 

\begin{thm}
\label{thm:main}
 Let $H \subseteq G$ be connected reductive groups and let $K$ be a closed subgroup of $H$. Suppose $p \ge d(G)$. 
 If $K$ is $H$-completely reducible, then $K$ is $G$-completely reducible.
\end{thm}

\noindent Theorem~\ref{thm:main} generalises \cite[Thm.\ 1.3]{BHMR}, which is the special case when $G= \GL(V)$.  We note that Theorem~\ref{thm:main} is false without the 
bound on $p$: e.g., see \cite[Ex.~3.44]{BMR} or Example \ref{ex:notHsigmacr}.

To prove Theorem~\ref{thm:main} we use Richardson's theory of reductive pairs (see Definition~\ref{def:redpair}) to reduce to the case $G= \GL(V)$, then we apply \cite[Thm.\ 1.3]{BHMR}.  We require a careful analysis of when the Lie algebra of $G$ admits a nondegenerate $\Ad$-invariant bilinear form, building on the discussions in \cite{rich2} and \cite{slodowy} (see also \cite[Sec.\ 3.5]{BMR}).  We establish the results we need in Theorem~\ref{thm:redpairs}.

Our second main result involves the concept of saturation due to Serre.  Let $h(G)$ and $\widetilde{h}(G)$ be as defined in Section~\ref{sec:coxeter} (if $G$ is simple then $h(G)$ is the Coxeter number of $G$).  If $p \geq h(G)$ then for any unipotent $u\in G$ and any $t\in k$ one obtains a unipotent element $u^t$ of $G$ using versions of the matrix exponential and logarithm maps; see Section~\ref{sec:sat} for details.  If $H$ is a closed subgroup of $G$ then we call $H$ \emph{saturated} if $u^t\in H$ for every unipotent $u\in H$ and every $t\in k$.  We denote by $H^{\rm sat}$ the smallest closed saturated subgroup that contains $H$.  Saturation was used by Serre in his study of $G$-complete reducibility in \cite{serre1} and \cite{serre2} (see Theorem~\ref{thm:saturation}).

\begin{thm}
\label{thm:sat_cr}
 Let $G$, $H$ and $K$ be as in Theorem~\ref{thm:main}.   
 Suppose $p \ge d(G)$ and $p\geq \widetilde{h}(G)$.
If $K$ is $H$-completely reducible then $K^\sat$ is $H^\sat$-completely reducible.
\end{thm}

\noindent (Note that $K$ is $G$-cr under the hypotheses of Theorem~\ref{thm:sat_cr}, by Theorem~\ref{thm:main}.)  For a direct application of saturation to Questions~\ref{qn:ascent} and \ref{qn:descent}, see Proposition~\ref{prop:HcrvsGcr}.

In order to prove Theorem~\ref{thm:sat_cr} we establish some results on saturation that are of interest in their own right.  These include Remark~\ref{rem:regularsaturated}, Lemma~\ref{lem:products} and Corollary~\ref{cor:product_proj}.

We also give some versions of our main results for finite groups of Lie type.  Recall that a \emph{Steinberg endomorphism} of $G$ is a surjective morphism $\sigma:G\to G$ such that the 
corresponding fixed point subgroup $G_\sigma :=\{g \in G \mid \sigma(g) = g\}$ of $G$ is finite. The latter are the 
\emph{finite groups of Lie type}; see 
Steinberg \cite{steinberg:end} for a detailed discussion.
The set of all Steinberg endomorphisms of $G$ is a subset of the set of all 
isogenies $G\rightarrow G$ (see \cite[7.1(a)]{steinberg:end}) 
which encompasses in particular all generalized Frobenius endomorphisms, 
i.e., endomorphisms of $G$ some power of which are Frobenius endomorphisms 
corresponding to some $\BBF_q$-rational structure on $G$.  If $H$ is a connected reductive $\sigma$-stable subgroup $H$ of $G$ then $\sigma$ is also a Steinberg endomorphism for $H$ with finite fixed point subgroup  $H_\sigma = H \cap G_\sigma$, \cite[7.1(b)]{steinberg:end}.

\begin{cor}
	\label{cor:finite}
	 Let $H \subseteq G$ be connected reductive groups. 
 Let $\sigma\colon G\to G$ be a Steinberg endomorphism that stabilises $H$.  
 Suppose $p \ge d(G)$.
 Then the fixed point subgroup $H_\sigma$ is $G$-completely reducible.
\end{cor}

\begin{thm}
	\label{thm:nori}
		Suppose $G$ is simple.
	Let $p \ge \widetilde{h}(G)$. 	
	Let $\sigma$ be a standard Frobenius endomorphism of $G$ and let $H$ be a connected reductive, $\sigma$-stable, and saturated  subgroup of $G$. 
	Then $(H_\sigma)^\sat = H$.
\end{thm}
 
\noindent Theorem~\ref{thm:nori} generalises a theorem of Nori \cite[Thm.~B(2)]{nori}; Nori's result is the special case of Theorem~\ref{thm:nori} for $G = \GL_n$ and $\sigma =\sigma_p$ the standard Frobenius endomorphism of $G$ raising the matrix coefficients to the $p^{\rm th}$ power.  In that context Proposition~\ref{prop:frobsat} is of independent interest which says that saturation is compatible with standard Frobenius endomorphisms.

\bigskip
Section~\ref{sec:prelim} contains some preliminary material and Section~\ref{sec:redpair} deals with reductive pairs.  We prove Theorem~\ref{thm:main} and 
Corollary~\ref{cor:finite} in Section 
\ref{sec:proofs}.  Results on saturation, including the proof of Theorem~\ref{thm:sat_cr}, are treated in Section \ref{sec:sat}.  In Section \ref{sec:sat-finite} we study saturation for finite subgroups of Lie type. Here we prove Theorem~\ref{thm:nori}, among other results.
Finally, Section \ref{sec:semisat} then explores the connection between saturation and the concept of a semisimplification of a subgroup of $G$ from \cite{BMR:semisimplification}.

\section{Preliminaries}
\label{sec:prelim}

Throughout, we work over an algebraically closed field $k$ of characteristic $p\geq 0$. All affine varieties are considered over $k$ and are identified with their $k$-points.

A linear algebraic group $H$ over $k$ has identity component $H^\circ$; if $H=H^\circ$, then we say that $H$ is \emph{connected}.
We denote by $R_u(H)$ the \emph{unipotent radical} of $H$; if $R_u(H)$ is trivial, then we say $H$ is \emph{reductive}.

Throughout, $G$ denotes a connected reductive linear algebraic group over $k$. All subgroups of $G$ that are considered are closed.

\subsection{Good and very good primes}
Suppose $G$ is simple.
Fix a Borel subgroup $B$ of $G$ containing a maximal torus $T$.
Let $\Phi = \Phi(G,T)$ be the root system of $G$ with respect to $T$,  
let $\Phi^+ = \Phi(B,T)$ be the set of positive roots of $G$, and let 
$\Sigma = \Sigma(G, T)$   
be the set of simple roots of the root system $\Phi$ of $G$ defined by $B$. 
For $\beta \in \Phi^+$ write 
$\beta = \sum_{\alpha \in \Sigma} c_{\alpha\beta} \alpha$
with $c_{\alpha\beta} \in \mathbb N_0$.
A prime $p$ is said to be \emph{good} for $G$
if it does not divide $c_{\alpha\beta}$ for any $\alpha$ and $\beta$.
A prime $p$ is said to be \emph{very good} for $G$ if 
$p$ is a good prime for $G$ and in case $G$ is of type 
$A_n$, then $p$ does not divide $n+1$.
For $G$ reductive $p$ is \emph{good} (\emph{very good}) for $G$ if $p$ is good  (very good) for every simple component of $G$.

\subsection{Limits and parabolic subgroups}
	Let $\phi:k^*\to X$ be a morphism of algebraic varieties. We say $\lima \phi(a)$ exists if there is a morphism $\hat{\phi}:k\to X$ (necessarily unique) whose restriction to $k^*$ is $\phi$; if the limit exists, then we set $\lima \phi(a)=\hat{\phi}(0)$.
As a direct consequence of the definition we have the following:

\begin{rem}\label{rem:mophisms_lim}
	If $\phi:k^*\to X$ and $h:X\to Y$ are morphisms of varieties and $x:=\lima\phi(a)$ exists then $\lima(h\circ\phi)(a)$ exists, and $\lima(h\circ\phi)(a)=h(x)$.
\end{rem}

For an algebraic group $G$ we denote by $Y(G)$ the set of cocharacters of $G$. For  
$\lambda \in Y(G)$ we define 
$P_\lambda:=\{g\in G\mid \lima \lambda(a)g\lambda(a)^{-1}\ \text{exists}\}$.

\begin{lem}[{\cite[Lem.~2.4]{BMR}}]
	\label{lem:Plambda}
	Given a parabolic subgroup $P$ of $G$ and any Levi subgroup $L$ of $P$, there exists a $\lambda\in Y(G)$ such that the following hold:
	\begin{enumerate}[label=\rm(\roman*)]
		\item $P=P_\lambda$.
		\item $L=L_\lambda:=C_G(\lambda(k^*))$.
		\item The map $c_\lambda:P_\lambda\to L_\lambda$ defined by 
		\[c_\lambda(g):=\lima \lambda(a)g\lambda(a)^{-1}\]
		is a surjective homomorphism of algebraic groups. Moreover, $L_\lambda$ is the set of fixed points of $c_\lambda$ and $R_u(P_\lambda)$ is the kernel of $c_\lambda$.
	\end{enumerate}
	Conversely, given any $\lambda\in Y(G)$ the subset $P_\lambda$ defined above is a parabolic subgroup of $G$, $L_\lambda$ is a Levi subgroup of $P_\lambda$ and the map $c_\lambda$ as defined in (iii) has the described properties.
\end{lem}

\subsection{\texorpdfstring{$G$}{G}-complete reducibility, products and epimorphisms}
\label{sub:isogenies}

Let $f:G_1\rightarrow G_2$ be a homomorphism of algebraic groups.  
We say that $f$ is \emph{non-degenerate} provided 
$({\rm ker}f)^\circ$ is a torus, see \cite[Cor.\ 4.3]{serre2}.
In particular, $f$ is non-degenerate if $f$ is an isogeny.

\begin{lem}[{\cite[Lem.~2.12]{BMR}}]
	\label{lem:epimorphisms}
	Let $G_1$ and $G_2$ be reductive groups.
	\begin{enumerate}[label=\rm(\roman*)]
		\item
		Let $H$ be a 
		closed subgroup of $G_1\times G_2$.  
		Let $\pi_i:G_1\times G_2\rightarrow G_i$ 
		be the canonical projection for $i = 1,2$.  Then $H$ 
		is $(G_1\times G_2)$-cr  
		if and only if $\pi_i(H)$ is $G_i$-cr 
		for $i=1,2$.
		\item
		Let $f:G_1\rightarrow G_2$ be an epimorphism.  
		Let $H_1$ and $H_2$ be closed subgroups of $G_1$ and $G_2$, respectively.
		\begin{enumerate} 
			\item[(a)] 
			If $H_1$ is $G_1$-cr,  
			then $f(H_1)$ is $G_2$-cr. 
			\item[(b)] 
			If $f$ is non-degenerate, then $H_1$ is $G_1$-cr  
			if and only if $f(H_1)$ is $G_2$-cr, 
			and $H_2$ is $G_2$-cr  
			if and only if 
			$f^{-1}(H_2)$ is $G_1$-cr. 
		\end{enumerate}
	\end{enumerate}
\end{lem}

\subsection{Complete reducibility versus reductivity}

Any $G$-completely reducible subgroup $H$ of $G$ is reductive \cite[Prop.~4.1]{serre2}.  The converse is true in characteristic 0 --- in particular, the answer to both Questions~\ref{qn:ascent} and \ref{qn:descent} is yes in this case --- but is false in positive characteristic: for instance, a non-trivial finite unipotent subgroup of $G$ is reductive but can never be $G$-cr \cite[Prop.~4.1]{serre2}.  The situation is somewhat nicer for connected $H$: the converse is true if $p$ is sufficiently large.  To be precise, we have the following theorem due to Serre.

\begin{thm}[{\cite[Thm.\ 4.4]{serre2}}]
\label{thm:ag}
 Suppose $p \geq a(G)$ and $(H:H^\circ)$ is prime to $p$.  Then $H^\circ$ is reductive if and only if $H$ is $G$-completely reducible.
\end{thm}

Here the invariant $a(G)$ of $G$ is defined as follows \cite[\S 5.2]{serre2}.
For $G$ simple, set $a(G) = \rk(G) +1$, 
where $\rk(G)$ is the rank of $G$.
For $G$ reductive, let $a(G) = \max(1, a(G_1), \ldots, a(G_r))$, 
where $G_1, \ldots, G_r$ are the simple components of $G$.  
In the special case $G= \GL(V)$ we have $a(G)= \dim(V)$, and a subgroup $H$ of $G$ is $G$-cr if and only if $V$ is a semisimple $H$-module.  We recover a basic result of Jantzen \cite[Prop.\ 3.2]{jantzen0}: if $\rho\colon H\to \GL(V)$ is a representation of a connected reductive group $H$ and $p\geq \dim(V)$ then $\rho$ is completely reducible.  The finite unipotent example above shows that we cannot expect Theorem~\ref{thm:ag} to carry over completely, even with extra restrictions on $p$.  Even when $H$ is connected, $H$ can fail to be $G$-cr if $p$ is small.

\begin{rem}
 Note that if we assume that $K$ is connected then Theorem~\ref{thm:main} follows immediately from Theorem~\ref{thm:ag}, since $d(G)\geq a(G)$.  We need a more elaborate proof with a worse bound on $p$, as we do not wish to place any restrictions on $(K:K^\circ)$; Theorem \ref{thm:main} is only of independent interest when $(K:K^\circ)$ is not prime to $p$.  For an application in such an instance, see Corollary~\ref{cor:finite}.
\end{rem}

\subsection{Coxeter numbers}
\label{sec:coxeter}

The invariant $h(G)$ denotes the maximum of the \emph{Coxeter numbers} of the simple components of $G$ \cite[(5.1)]{serre2}, and we define $\widetilde{h}(G)$ to be $h(G)$ if $[G,G]$ is simply connected and $h(G)+ 1$ otherwise.  
Recall that if $G$ is simple, we have $h(G) + 1= \dim(G) / \rk(G)$; the values of $h(G)$ for the
various Dynkin types are as follows:

\begin{center}
	\begin{tabular}{c|*{8}{c}}
		$G$ & $A_n \ (n\ge1)$ & $B_n , C_n \ (n\ge2)$ & $D_n \ (n\ge4)$ & $E_6$ & $E_7$ & $E_8$ & $F_4$ & $G_2$ \\ \hline
		$h(G)$ & $n+1$ & $2n$ & $2n-2$ & $12$ & $18$ & $30$ & $12$ & $6$
	\end{tabular}
\end{center}

We thus have
\begin{equation}
\label{eqn:inequalities}
 a(G) \le h(G) \le d(G)
\end{equation}
for any $G$.  Note, however, that we can have $\widetilde{h}(G)> d(G)$: for instance, take $G= \PGL_n$.

\section{\texorpdfstring{$G$}{G}-complete reducibility, separability and reductive pairs}
\label{sec:redpair}

Now we consider the 
interaction of subgroups of $G$ with the Lie algebra $\Lie G = \gg$ of $G$.
Much of this material is taken from \cite{BMR}.

\begin{defn}[{\cite[Def.~3.27]{BMR}}]
	\label{def:separable} 
	For a closed subgroup $H$ of $G$, the Lie subalgebra $\Lie(C_G(H))$ is contained in  $\cc_\gg(H)$, the fixed-point space of $H$ on $\gg$ in the adjoint action. 
	In case of equality (that is, if the scheme-theoretic centralizer of $H$ in $G$ is smooth), we say that $H$ is \emph{separable in $G$}; else  $H$ is \emph{non-separable in $G$}.  (See \cite[Rem.~3.32]{BMR} for an explanation of the terminology.)
\end{defn}

Of central importance is the following observation.

\begin{exmp}[{\cite[Ex.~3.28]{BMR}}]
	\label{exmp:allsep}
	Any closed subgroup $H$ of $G=\GL(V)$ is separable in $G$.  
\end{exmp}

\begin{defn}
	\label{def:redpair} 
	Following Richardson \cite{rich2}, we call 
	$(G,H)$ a \emph{reductive pair} provided that $H$ is a reductive subgroup of 
	$G$ and under the adjoint action, $\Lie G$ decomposes as a direct sum
	\[
	\Lie G = \Lie H \oplus \mm, 
	\]
	for some $H$-submodule $\mm$.
\end{defn}

For a list of examples of reductive pairs we refer 
to P.\ Slodowy's article  \cite[I.3]{slodowy}.  
For further examples, see \cite[Ex.~3.33, Rem.~3.34]{BMR}.

Our application of reductive pairs to $G$-complete reducibility goes via the following result.

\begin{prop}[{\cite[Cor.~3.36]{BMR}}]
	\label{prop:red-pair-gl}
	Suppose that $(\GL(V),H)$ is a reductive pair
	and  $K$ is a closed subgroup of  $H$. 
	If  $V$ is a semisimple $K$-module, then
	$K$ is $H$-completely reducible.
\end{prop}

We look first at the special case of the adjoint representation. 

\begin{exmp}[{\cite[Ex.~3.37]{BMR}}]
	\label{ex:adjoint}
	Let $H$ be a simple group of adjoint type and let 
	$G = \GL(\Lie H)$.
	We have a symmetric non-degenerate Ad-invariant bilinear form 
	on $\Lie G \cong \End (\Lie H)$ 
	given by the usual trace form 
	and its restriction to $\Lie H$ is just the Killing form of $\Lie H$.
	Since $H$ is adjoint and $\Ad$ is a closed embedding, 
	$\ad : \Lie H \to \Lie \Ad(H)$ is surjective. 
	Thus it follows from the arguments in \cite[Rem.~3.34]{BMR} 
	that if the Killing form of $\Lie H$ 
	is non-degenerate, then $(G,H)$ is a reductive pair.
	
	Suppose first that $H$ is a simple classical group of adjoint type
	and $p > 2$.
	The Killing form 
	is non-degenerate for ${\mathfrak {sl}}(V)$, ${\mathfrak {so}}(V)$, or
	${\mathfrak {sp}}(V)$  if and only if $p$  does not
	divide  $2\dim V$, $\dim V-2$, or $\dim V+2$, respectively,
	cf.\ \cite[Ex.~Ch.~VIII, \S 13.12]{bou2}.  
	In particular, for  $H$ adjoint of type 
	$A_n$, $B_n$, $C_n$, or $D_n$, the Killing form 
	is non-degenerate if $p > 2$  and $p$ does not
	divide $n+1$, $2n-1$, $n+1$, or $n-1$, respectively. 
	
	Now suppose that $H$ is a simple exceptional group of adjoint type. 
	If $p$ is good for $H$, then
	the Killing form of $\Lie H$ is non-degenerate; this 
	was noted by Richardson (see \cite[\S 5]{rich2}).
	Thus if $p$ satisfies the appropriate condition, then 
	$(\GL(\Lie H),H)$ is a reductive pair and Proposition~\ref{prop:red-pair-gl}
	applies.
\end{exmp}

The next result is new, and gives a more general characterisation of reductive pairs.

\begin{thm}
	\label{thm:redpairs}
For a simply connected simple algebraic group $G$ in characteristic $p \ge 0$, consider the following conditions:
\begin{enumerate}[label=\rm(\roman*)]
	\item $(\GL(V),\rho(G))$ is a reductive pair, where $\rho : G \to \GL(V)$ is a non-trivial representation of least dimension; \label{minred}
	\item $(\GL(V),\rho(G))$ is a reductive pair, for some non-trivial irreducible representation $\rho : G \to \GL(V)$; \label{existred}
	\item $p$ is very good for $G$; \label{pvg}
	\item $(\GL(\Lie(G)),\Ad(G))$ is a reductive pair; \label{adred}
	\item the Killing form on $\Lie(G)$ is non-degenerate; \label{nondeg}
	\item $p$ is very good for $G$ and, if $G$ has classical type, then $p \nmid e(G)$ as follows: \label{supergood}
	\begin{center}
		\begin{tabular}{c|*{4}{c}}
			$G$ & $A_n$ & $B_n$ & $C_n$ & $D_n$ \\ \hline
			$e(G)$ & $2$ & $2n-1$ & $n+1$ & $n-1$
		\end{tabular}
	\end{center}
\end{enumerate}
Then \ref{minred} $\Leftrightarrow$ \ref{existred} $\Leftrightarrow$ \ref{pvg} $\Leftarrow$ \ref{adred} $\Leftarrow$ \ref{nondeg} $\Leftrightarrow $\ref{supergood}. 

For $G$ of exceptional type, all these conditions are equivalent.
\end{thm}

\begin{proof}
	It is clear that \ref{minred} implies \ref{existred}. If \ref{existred} holds, then every subgroup of $\rho(G)$ is separable, since every subgroup of $\GL(V)$ is separable, by Example \ref{exmp:allsep},  and this descends through reductive pairs. This means that $p$ is pretty good for $\rho(G)$, see \cite[Def.\ 2.11]{herpel}, which is the same as $p$ being very good for $\rho(G)$ since $G$ is simple. Note that $p$ being very good is insensitive to the isogeny type of $G$, so \ref{existred} implies \ref{pvg}.
	
	Next, the implication \ref{pvg} $\Rightarrow$ \ref{minred} for $G$ of type $B_n$, $C_n$ and $D_n$ is \cite[Lem.~5.1]{rich2}. For $\SL(V)$, it is well-known that the traceless matrices $\sl(V)$ and the scalar matrices $\cc(\gl(V))$ are the only proper, nonzero $\GL(V)$-submodules (and $\SL(V)$-submodules) of $\gl(V)$.
	Now \ref{pvg} means that $p$ is coprime to $\dim V$, which implies that these submodules intersect trivially, so that $\sl(V)$ is complemented by $\cc(\gl(V))$. So \ref{minred} holds in type $A_n$.
	
	It remains to consider the exceptional cases $(G, V) = (G_2, V_7)$, $(F_4, V_{26})$, $(E_6, V_{27})$, $(E_7, V_{56})$ and $(E_8, V_{248})$ where $V_j$ is a minimal-dimensional $G$-module in each case. Now $\Lie(\GL(V)) = V_{j} \otimes V_{j}^{\ast}$ contains $\Lie(G)$ as a submodule, which is irreducible since \ref{pvg} holds.
	The known weights of low-dimensional irreducible $G$-modules, given for instance in \cite{luebeck}, now allow one to calculate the weights of $V_j \otimes V_j^{\ast}$ and determine its $G$-composition factors.
	In each case, it transpires that $\Lie(G)$ is in fact the unique $G$-composition factor with a particular highest weight. Since $\Lie(G)$ is a submodule and $V_j \otimes V_j^{\ast}$ is self-dual, $\Lie(G)$ also occurs as a quotient of this, and the kernel of this quotient map is then a complement to $\Lie(G)$, so that \ref{minred} indeed holds.
	We illustrate the details in case $G$ is of type $G_2$ in Remark \ref{rem:redpair}(iv).
		
	This shows that \ref{minred}, \ref{existred} and \ref{pvg} are equivalent. Now it is clear that \ref{adred} $\Rightarrow$ \ref{existred}. The equivalence of \ref{nondeg} and \ref{supergood} 
	follows from Example \ref{ex:adjoint}. 
If \ref{nondeg} holds then the Killing form on $\Lie(G)$ is, up to a non-zero scalar, the restriction of the trace form on $\Lie(\GL(\Lie(G)))$, and hence $\Lie(G)$ has an orthogonal complement, so that \ref{adred} holds. Also, \ref{pvg} coincides with \ref{supergood} when $G$ has exceptional type, which shows that all the conditions are equivalent in this case.
\end{proof}

\begin{rems}
	\label{rem:redpair} 
	(i). For exceptional groups in good characteristic, one can also check, just as in the case of the minimal module, that $\Lie(G)$ is the unique $G$-composition factor of $\Lie(G) \otimes \Lie(G)^{\ast}$ having a particular high weight, so that (being a submodule) it is a direct summand; this gives an alternative direct proof that \ref{pvg} $\Rightarrow$ \ref{adred} in Theorem \ref{thm:redpairs} for exceptional $G$.
		
	(ii). For type $A_n$ when $p = 2 \nmid n+1$, so that $\Lie(G)$ is simple and self-dual but the Killing form vanishes, evidence suggests that $(\GL(\Lie(G)),G)$ is nevertheless a 
	reductive pair. For instance, if $n=2$ or $4$ and $G = \SL_{n+1}$, the module $V = \Lie(G)$ appears with multiplicity $2$ as a direct summand of $\gl(V) = V \otimes V^{\ast}$; in fact we have $\gl(V) \cong V \oplus V \oplus W$ for some indecomposable $G$-module $W$. With respect to the trace form on $\gl(V)$, the image of the embedding $\Lie(G) \to \gl(V)$ gives a totally isotropic subspace isomorphic to $V$, and the natural isomorphism $\gl(V)/V^{\perp} \to V^{\ast}$ then shows the existence of a second composition factor isomorphic to $V$. One can then use the self-duality of $V$ and $\gl(V)$ to deduce that $\Lie(G)$ and this second composition factor are each a direct summand of $\gl(V)$.
		
	(iii). For types $B_n$, $C_n$, $D_n$ with $p$ odd but dividing $2n-1$, $n+1$, $n-1$ respectively, considering such instances in rank up to $6$ suggests that $(\GL(\Lie(G)),G)$ is never a reductive pair. 
	In each case, it transpires that $\Hom(\Lie(G),\gl(\Lie(G))$ is $1$-dimensional. Thus $\gl(\Lie(G))$ has a unique $G$-submodule isomorphic to $\Lie(G)$, which turns out to lie in a self-dual indecomposable $G$-module direct summand of $\gl(\Lie(G))$ also having $\Lie(G)$ as its head.
	
	(iv). To illustrate the argument in the proof above, when $G$ has type $G_2$ and $p \neq 2,3$ the $G$-module $V_{7}$ is irreducible of highest weight $\lambda_2$, and the $49$-dimensional module $V_{7} \otimes V_{7}^{\ast}$ has high weights $0$, $\lambda_1$, $\lambda_2$ and $2\lambda_2$ when $p \neq 7$; or $0$, $0$, $\lambda_1$, $\lambda_2$ and $2\lambda_2$ when $p = 7$. In either case, we find a unique composition factor of high weight $\lambda_1$, which is $\Lie(G)$.
\end{rems}

\section{Proofs of Theorem~\ref{thm:main} and Corollary~\ref{cor:finite}}
\label{sec:proofs}

\begin{proof}[Proof of Theorem~\ref{thm:main}]
	Let $\pi\colon G\to G/Z(G)^\circ$ be the canonical projection.  Owing to  
	Lemma \ref{lem:epimorphisms}(ii)(b), 
	we can replace $G$ with $G/Z(G)^\circ$, so without loss we can assume that $G$ is semisimple.  Let $G_1,\ldots, G_r$ be the simple factors of $G$.  Multiplication gives an isogeny from $G_1\times\cdots \times G_r$ to $G$.  Again by Lemma \ref{lem:epimorphisms}(ii)(b), 
	we can replace $G$ with $G_1\times\cdots \times G_r$, so we can assume $G$ is the product of its simple factors.  By  
	Lemma \ref{lem:epimorphisms}(i)
	it is thus enough to prove the result when $G$ is simple and simply connected.
	Of course, as well as replacing $G$ with its (pre-)image under an isogeny, we also replace $H$ and $K$ with their (pre-)images under that isogeny along the way.
	
	First suppose $G$ is of type $A$.
	Then $K\subseteq H\subseteq \SL(V)\subseteq \GL(V)$.  By \cite[Thm.\ 1.3]{BHMR}, if $K$ is $H$-cr then $K$ is $\GL(V)$-cr, so $K$ is $\SL(V)$-cr and we’re done.
	Next suppose $G$ is not of type $A$.
	Then, as $p \ge d(G)$, $p$ is good for $G$. It follows from Theorem \ref{thm:redpairs} that $(\GL(V),G)$ is a reductive pair, where $V$ is an irreducible $G$-module of least dimension.
Since $p \ge d(G) = \dim V$,  $V$ is semisimple for $K$, thanks to \cite[Thm.\ 1.3]{BHMR}, and thus 
$K$ is $G$-cr, by Proposition~\ref{prop:red-pair-gl}.
\end{proof}

The following is an immediate consequence of 
\cite[III 1.19(a)]{springersteinberg}
and \cite[Thm.~1.3]{BMR:regular}.

\begin{lem}
	\label{lem:GsigmaGcr}
	Let $\sigma$ be a Steinberg endomorphism of $G$.
	Then $G_\sigma$ is $G$-irreducible.
\end{lem}

\begin{proof}[Proof of Corollary~\ref{cor:finite}]
 By Lemma \ref{lem:GsigmaGcr}, $H_\sigma$ is $H$-ir.  The result now follows from Theorem~\ref{thm:main}.
\end{proof}

Note that Corollary \ref{cor:finite} is false without the bound on $p$. See Example \ref{ex:notHsigmacr} for an instance 
when $H$ is $G$-cr but $H_\sigma$ is not, where 
$p = 3 < 8 = d(G)$.

Our next result gives a particular set of conditions on $H_\sigma$ to guarantee that $H_\sigma$ and  $H$ 
belong to the same parabolic subgroups and the same Levi subgroups of $G$.
Note that, if  $\sigma\colon H\to H$ is a Steinberg endomorphism of $H$, then $\sigma$ stabilises a maximal torus of $H$, \cite[Cor.~10.10]{steinberg:end}. 
Also, for $S$ a torus in $G$, we have $C_G(S) = C_G(s)$ for some  $s \in S$, see \cite[III Prop.~8.18]{borel}.

\begin{prop}
	\label{prop:HcrvsHsigmacr}
	Let $H \subseteq G$ be connected reductive groups. 
	Let $\sigma\colon G\to G$ be a Steinberg endomorphism that stabilises $H$  and a maximal  torus $T$  of $H$. Suppose
	\begin{enumerate}[label=\rm(\roman*)]
		\item   $C_G(T) = C_G(t)$, for some $t \in T_\sigma$,   and
		\item  $H_\sigma$ meets every $T$-root subgroup of $H$ non-trivially.  
	\end{enumerate}
	Then $H_\sigma$ and $H$ are contained in precisely the same parabolic subgroups of $G$, and the same Levi factors thereof. In particular, $H$ is $G$-completely reducible if and only if $H_\sigma$ is $G$-completely reducible; similarly, $H$ is $G$-irreducible if and only if $H_\sigma$ is $G$-irreducible.
\end{prop}

\begin{proof}
	First assume $H_\sigma\subseteq P$ for some parabolic subgroup $P$ of $G$. Then $t$ lies in some maximal torus of $P$. So we can find a $\lambda \in Y(G)$ such that $P= P_\lambda$ and $\lambda$ centralizes $t$. But then $\lambda$ centralizes $T$, by (i), so $T\subseteq P$. Now we have a maximal torus $T$ of $H$ and a non-trivial part of each $T$-root group of $H$ inside $P$, by (ii), so we can conclude that all of $H$ belongs to  $P$.  Similarly, if $H_\sigma\subseteq L_\lambda$ for some $\lambda \in Y(G)$, we get $H\subseteq L_\lambda$.  The reverse conclusions are obvious, since $H_\sigma\subseteq H$. 
\end{proof}

In the presence of the conditions in  
Proposition \ref{prop:HcrvsHsigmacr}
we can improve the bound in Corollary \ref{cor:finite} considerably;
the following is immediate from Theorem \ref{thm:ag} and Proposition \ref{prop:HcrvsHsigmacr}.  

\begin{cor}	
	\label{cor:HcrvsHsigmacr}
Suppose $G, H$ and $\sigma$ satisfy the hypotheses of Proposition \ref{prop:HcrvsHsigmacr}. Suppose in addition that $p\geq a(G)$. 
	Then $H_\sigma$ is $G$-completely reducible.
\end{cor}

Note that condition (ii) 
in Proposition \ref{prop:HcrvsHsigmacr} 
is automatically satisfied provided $\sigma$ induces a standard Frobenius endomorphism on $H$.
In that case Example \ref{ex:Tsigma} below demonstrates that condition (i)  above 
does hold generically. 
Nevertheless, Example \ref{ex:notHsigmacr} shows that
Proposition \ref{prop:HcrvsHsigmacr}
is false in general without condition (i) even when part (ii) is fulfilled.

The following example shows that the conditions in 
Proposition \ref{prop:HcrvsHsigmacr} do hold generically. 

\begin{ex}
	\label{ex:Tsigma}
	Let $\sigma_q \colon \GL(V) \to \GL(V)$ be a standard Frobenius endomorphism that stabilises the connected reductive subgroup $H$ of $\GL(V)$  and a maximal  torus $T$  of $H$.  
	Pick $l \in \BBN$ so that firstly all the different $T$-weights of $V$ are still distinct when restricted to $T_{\sigma_q^l}$ and secondly that there is a $t \in T_{\sigma_q^l}$, such that $C_{\GL(V)}(T) = C_{\GL(V)}(t)$.
	Then for every $n \ge l$, both conditions in Proposition \ref{prop:HcrvsHsigmacr} are satisfied for $\sigma = \sigma_q^n$. Thus there are only finitely many powers of $\sigma_q$ for which part (i) can fail.
	The argument here readily generalises to a Steinberg endomorphism of  a connected reductive $G$ which induces a generalised Frobenius morphism on $H$.
\end{ex}

In contrast to the setting in Example \ref{ex:Tsigma}, 
our next example demonstrates that  
the conclusion of Proposition \ref{prop:HcrvsHsigmacr} may fail, if 
condition (i) is not satisfied.
Consequently, the  conditions in  
Theorem~\ref{thm:main} and 
Corollary \ref{cor:finite} are needed in general.

\begin{ex}
	\label{ex:notHsigmacr}
	Let $p=3$, $q=9$ and $H = \SL_2$.  
	By Steinberg's tensor product theorem, the simple $H$-module $V = L(1+q+q^2)$  
	is isomorphic to $L(1)\otimes L(1)^{[2]}\otimes L(1)^{[4]}$, the superscripts denoting $p$-power twists.
	Thus, after identifying $H$ with its image in $G = \GL(V)$, we see that $H$ is $G$-cr. Let $\sigma = \sigma_q$ be the standard Frobenius on $G$. Then  $H_\sigma = \SL_2(9)$ is $H$-cr, by Lemma \ref{lem:GsigmaGcr}. 
	Now as an $H_\sigma$-module, $V$ is isomorphic to the $H$-module
	$L(1)\otimes L(1) \otimes L(1)$ which 
	admits the non-simple indecomposable Weyl module of highest weight $3$ 
	as a constituent. As the latter is not semisimple for $H_\sigma$,
	$V$ is not semisimple as an $H_\sigma$-module and so $H_\sigma$ is not $G$-cr. 
\end{ex}

\section{Saturation}
\label{sec:sat}

Let $u \in \GL(V)$ be unipotent of order $p$. Then there is a nilpotent element $\epsilon \in \End(V)$
with $\epsilon^p = 0$ such that $ u = 1 + \epsilon$. 
For $t \in \mathbb G_a$ we define $u^t$ by
\begin{equation}
	\label{eq:ut}
	u^t := (1 + \epsilon)^t = 1 + t\epsilon + \binom{t}{2} \epsilon^2 + \cdots + \binom{t}{p-1} \epsilon^{p-1},
\end{equation}	
see \cite{serre1}.
Then $\{u^t \mid t \in  {\mathbb G}_a\}$ is a closed connected subgroup  of $\GL(V)$
isomorphic to ${\mathbb G}_a$.

Following \cite{nori} and \cite{serre1}, a subgroup $H$ of $\GL(V)$ is \emph{saturated}  provided $H$ is closed and for any unipotent element $u$ of $H$ of order $p$ and any $t \in \BBG_a$ also $u^t$ given by \eqref{eq:ut} belongs to $H$.
The \emph{saturated closure}
$H^\sat$ of $H$ is the smallest saturated subgroup of $\GL(V)$ containing $H$.

We now recall a notion of saturation for 
arbitrary connected reductive groups which generalises the one just given for $\GL(V)$.
Suppose that $p \ge h(G)$.  
Then every unipotent element of $G$ has order $p$, see \cite{testerman}.
Let $u$ be a unipotent element of $G$. Then for $t \in \BBG_a$ there is a canonical ``$t^{\rm th}$ power'' $u^t$ of $u$ such that the map $t \mapsto u^t$ defines a homomorphism of the additive group $\BBG_a$ into $G$. 
We recall 
some results from \cite{serre1} and \cite[\S 5]{serre2}.

 Let $\UU$ be the subvariety of $G$ consisting of all unipotent elements of $G$ and let $\NN$ be the subvariety of $\Lie(G)$ consisting of all nilpotent elements of $\Lie(G)$.
 Fix a maximal torus $T$ of $G$, a Borel subgroup $B$ of $G$ containing $T$ and let $U$ be the unipotent radical of $B$.
 Since $p \ge h(G)$ and because the nilpotency class of $\Lie(U)$ is at most $h(G)$, 
 we can view $\Lie(U)$  
 as an algebraic group with multiplication given by the Baker-Campbell-Hausdorff formula (see \cite[Ch.~II \S 6]{bou}).
     
 Let $\Phi = \Phi(G,T)$ be the root system of $G$ with respect to $T$. For $\alpha$ a root in $\Phi$, let $x_\alpha : \BBG_a \to U_\alpha$ be a parametrization of the root subgroup $U_\alpha$ of $G$. Let $X_\alpha := \frac{d}{ds}(x_\alpha(s))|_{s=0}$ be a canonical generator of $\Lie(U_\alpha)$.
 Further, by $\Aut(G)$ we denote the group of algebraic automorphisms of $G$.
 We begin with the following result due to Serre; for a detailed proof, see \cite[\S 6]{BDP}.  Recall that we define $\widetilde{h}(G)$ to be $h(G)$ if $[G,G]$ is simply connected and $h(G)+1$ otherwise.
 
\begin{thm}[{\cite[Thm.~3]{serre1}}]
 	\label{thm:log}
 		Let $p \ge \widetilde{h}(G)$. 
 		There is a unique isomorphism of varieties $\log : \UU \to \NN$ such that the following hold:
\begin{enumerate}[label=\rm(\roman*)]
 		\item  $\log (\sigma u) = d \sigma (\log u)$ for any $\sigma \in \Aut(G)$ and any $u \in \UU$;
 		\item  the restriction of $\log$ to $U$ defines an isomorphism of algebraic groups
 		$U \to \Lie(U)$ whose tangent map is the identity on $\Lie(U)$;
 		\item  $\log(x_\alpha(t)) = t X_\alpha$ for any $\alpha \in \Phi$ and any $t \in \BBG_a$. 
 	\end{enumerate}
\end{thm}
 
 Let $\exp : \NN \to \UU$ be the inverse morphism to $\log$. We then define 
 \begin{equation}
 	\label{eq:sat}
 	 u^t := \exp(t \log u),
 \end{equation}
 for any $u \in \UU$ and any $t \in \BBG_a$.

\begin{defn}[{\cite{nori}, \cite{serre1}}]
	\label{def:saturation}
	Let $p \ge \widetilde{h}(G)$.  A subgroup $H$ of $G$ is \emph{saturated (in $G$)}  provided $H$ is closed and for any unipotent element $u$ of $H$ and any $t \in \BBG_a$ also $u^t$ belongs to $H$.
	For a subgroup $H$, its \emph{saturated closure}
	$H^\sat$ is the smallest saturated subgroup of $G$ containing $H$.	
\end{defn}

We give various fairly straightforward consequences of Theorem \ref{thm:log}. 
The third is already recorded in \cite{serre1} for centralizers of subgroups of $G$.
\begin{cor}
	\label{cor:centralizers}
 		Let $p \ge \widetilde{h}(G)$. Let $\sigma\in \Aut(G)$.
		Then the following hold:
			\begin{enumerate}[label=\rm(\roman*)]
			\item $\sigma(u^t) = \sigma(u)^t$ for any $u\in \UU$ and $t\in \BBG_a$;
			\item if $H$ is a $\sigma$-stable subgroup of $G$, so is $H^\sat$;  
			\item for $S$ a subgroup of $\Aut(G)$,  $C_G(S)$ is saturated in $G$. 
			\end{enumerate}
\end{cor}

\begin{proof}
(i). Since $\exp$ is the inverse to $\log$, Theorem \ref{thm:log}(i) gives $\sigma(\exp(X)) = \exp(d\sigma (X))$ for all $X\in \NN$.
Hence for any $u\in \UU$ and $t\in \BBG_a$,
$$
\sigma(u^t) = \sigma(\exp(t \log u)) = \exp(d\sigma(t \log(u))) = \exp(td\sigma(\log u)) = \exp(t\log \sigma(u)) = \sigma(u)^t. 
$$
(ii). If $H$ is $\sigma$-stable and $M$ is any saturated subgroup of $G$ containing $H$,
then so is $\sigma(M)$:
for, if $u\in \sigma(M)$ is unipotent, then $u = \sigma(v)$ for some 
$v\in M$ unipotent.
Then $u^t = \sigma(v)^t = \sigma(v^t) \in \sigma(M)$ also, by (i) and the fact that $M$ is saturated in $G$.
Hence $H^\sat$, the unique smallest saturated subgroup of $G$  containing $H$, must also be $\sigma$-stable.

Part (iii) is immediate by (i). 
\end{proof}

As particular instances of Corollary \ref{cor:centralizers}(iii), we note that centralizers of graph automorphisms of $G$ are saturated, 
and also Levi subgroups of parabolic subgroups of $G$ are saturated, since they arise as centralizers of tori in $G$.

One can use Theorem \ref{thm:log} directly to show that parabolic subgroups of $G$ are saturated.  
The following proof instead uses the language of cocharacters, which also allows us to observe that 
the process of ``taking limits along cocharacters'' commutes with saturation.

\begin{prop}\label{prop:P_saturated}
	Let $\lambda\in Y(G)$ and let $P=P_\lambda$. Then for $u\in P$ unipotent and $v:=\limau$, we have $\limaut=v^t$. In particular, $P$ is saturated.  
\end{prop}

\begin{proof}
	Observe that for any $t\in \BBG_a$ the map $h_t:\UU \to \UU$ given on points by $h_t(u):= u^t = \exp(t\log(u))$ is an isomorphism of varieties, by Theorem \ref{thm:log}. Furthermore, by Corollary \ref{cor:centralizers}(i) we have for any $a\in k^*$
	\[\left(\lambda(a)u\lambda(a)^{-1}\right)^t=\lambda(a)u^t\lambda(a)^{-1}.\]
	The result now follows from Remark \ref{rem:mophisms_lim} and elementary limit calculations.
\end{proof}

The next result is a slight refinement of a theorem due to Serre.
 
  \begin{thm}[{\cite[Property 2, Thm.~4]{serre1}}]
 	\label{thm:hGhH}
 	Let $p\geq \widetilde{h}(G)$.  Let $H$ be a saturated connected reductive subgroup of $G$, and suppose $p\geq \widetilde{h}(H)$.  Then for any $u \in H$ unipotent, the element $u^t$, with respect to $H$, coincides with $u^t$, with respect to $G$: that is, saturation in $H$ coincides with saturation in $G$. 
 \end{thm}

\begin{rem}
 It follows from \cite[Thm.~4]{serre1} that if $H$ is any connected reductive subgroup of $G$ --- not necessarily saturated --- then $h(H)\leq h(G)$.  It can happen, however, that $\widetilde{h}(H)> \widetilde{h}(G)= p$.  For instance, let $p$ be an odd prime and let $G= \SL_p\times \SL_p$.  Let $M$ be a semisimple subgroup of $\SL_p$ such that $M$ is {\bf not} simply connected, and let $H= \SL_p\times M$.  Now $\widetilde{h}(G)= h(G)= h(\SL_p)= p$ and $h(H)\geq h(\SL_p)= p$; but $H$ is not simply connected, so $\widetilde{h}(H)= h(H)+ 1\geq p+1> p= \widetilde{h}(G)$.  Because of this we have added the hypotheses that $p\geq \widetilde{h}(G)$ and $p\geq \widetilde{h}(H)$ to Theorem~\ref{thm:hGhH}, although they were not stated explicitly in \cite[Property 2]{serre1}.  The proof of \emph{loc.~cit.} still holds.
\end{rem}

\begin{rem}\label{rem:regularsaturated}
Suppose $p\ge \widetilde{h}(G)$ and let $H$ be a connected reductive subgroup of $G$ normalized by some maximal torus $T$ of $G$.
We claim that $H$ is saturated in $G$.
Since $H$ and $HT$ contain the same unipotent elements, $H$ is saturated if and only if $HT$ is saturated, so we may assume that $H$ contains $T$.
By Corollary~\ref{cor:centralizers}(iii) and Theorem~\ref{thm:hGhH} there is no harm in passing to a minimal Levi subgroup of $G$ containing $H$, so we may assume that $H$ has maximal semisimple rank in $G$.
Since $p\ge \widetilde{h}(G)$, $p$ is good for $G$, and hence $H$ arises as the centralizer $H = C_G(s)^\circ$ in $G$ of some semisimple element $s$ of $G$, thanks to  
Deriziotis' criterion; cf.\ \cite[\S 2.15]{Humphreys}.
Thus $H$ is saturated by Corollary \ref{cor:centralizers}(iii).

Note that this result applies to any connected normal subgroup of $G$, so in particular to the simple factors of $G$.
\end{rem}

Next we show that saturation is compatible with direct products, in the following sense.

\begin{lem}
	\label{lem:products}
 Suppose $G= G_1\times\cdots\times G_r$, where each $G_i$ is connected and reductive.  Suppose $p\geq \widetilde{h}(G)$.  Then $p\geq \widetilde{h}(G_i)$ and $G_i$ is saturated for $1\leq i\leq r$.  Moreover, if $u_i\in G_i$ is unipotent for $1\leq i\leq r$ then
 $$ \log(u_1\cdots u_r)= \log_1(u_1)+\dots+ \log_r(u_r), $$
 where $\log$ denotes the logarithm map for $G$ and $\log_i$ denotes the logarithm map for $G_i$.
\end{lem}

\begin{proof}
 For each $i$, any simple factor of $G_i$ is also a simple factor of $G$ and $[G_i, G_i]$ is simply connected if $[G,G]$ is simply connected, so $\widetilde{h}(G)\geq \widetilde{h}(G_i)$ and the first assertion follows.  Each $G_i$ is normal in $G$ and hence is saturated by Remark~\ref{rem:regularsaturated}.  Theorem~\ref{thm:hGhH} implies that if $u_i\in G_i$ is unipotent then $\log_i(u_i)= \log(u_i)$.  If $X$ and $X'$ are commuting nilpotent elements of $\Lie(G)$ then the Baker-Campbell-Hausdorff product of $X$ and $X'$ is just $X+ X'$, so $\exp(X+ X')= \exp(X)\exp(X')$.  The final assertion now follows easily.
\end{proof}

\begin{cor}
\label{cor:product_proj}
 Assume the hypotheses of Lemma~\ref{lem:products}, and let $\pi_i\colon G\to G_i$ be the canonical projection.  If $H$ is a saturated subgroup of $G$ then $\pi_i(H)$ is a saturated subgroup of $G_i$.
\end{cor}

\begin{proof}
 This follows immediately from Lemma~\ref{lem:products}.
\end{proof}
		 
\begin{prop}[{\cite[Prop.\ 5.2]{serre2}}]
	\label{prop:saturation-index}
	If $H$ is saturated in $G$, then 
	$(H: H^\circ)$ is prime to $p$.
\end{prop}

\begin{thm}[{\cite[Thm.\ 5.3]{serre2}}]
	\label{thm:saturation}
	Let $p \ge \widetilde{h}(G)$. For a closed subgroup  $H$ of $G$,
	the following are equivalent:
	\begin{enumerate}[label=\rm(\roman*)]
			\item  $H$ is $G$-completely reducible;
			\item  $H^\sat$ is $G$-completely reducible;
			\item  $(H^\sat)^\circ$ is reductive. 
		\end{enumerate}
\end{thm}

The equivalence between (i) and (ii) stems from the fact that both parabolic and Levi subgroups of $G$ are saturated.
Since $\widetilde{h}(G) \ge a(G)$ by \eqref{eqn:inequalities}, the equivalence between (ii) and (iii) is an immediate consequence of Theorem \ref{thm:ag} 
and Proposition \ref{prop:saturation-index}.

\begin{prop}
	\label{prop:HcrvsGcr}
	Let $K \subseteq H$ be closed subgroups of $G$ with $H$ connected reductive and saturated in $G$.  Suppose $p \ge \widetilde{h}(G)$ and $p \ge \widetilde{h}(H)$.
	Then $K$ is $H$-completely reducible if and only if $K$ is $G$-completely reducible.
\end{prop}

\begin{proof}
	By Theorem \ref{thm:saturation}, $K$ is $H$-cr if and only if $(K^\sat)^\circ$ is reductive, where we saturate in $H$.  But saturation in $H$ is the same as saturation in $G$ thanks to Theorem \ref{thm:hGhH}, so $(K^\sat)^\circ$ is reductive if and only if $K$ is $G$-cr, again by Theorem \ref{thm:saturation}. 
\end{proof}

Note that both implications in the equivalence in Proposition \ref{prop:HcrvsGcr} may fail if $p < h(G)$, e.g., see \cite[Ex.~3.45]{BMR} and \cite[Prop.~7.17]{BMRT}. 

For ease of reference, we recall a 
connectedness result for $H^\sat$ from 
\cite[Cor.~4.2]{BHMR}.

\begin{rem}
	\label{rem:Hsatconnected}
	Let $p \ge \widetilde{h}(G)$. If $H$ is a closed connected subgroup of $G$, then so is $H^\sat$.
	For, consider the subgroup $M$ of $G$ generated by 
	$H$ and the closed connected subgroups
	$\{u^t \mid t \in  {\mathbb G}_a\} \cong {\mathbb G}_a$  of $G$
	for each unipotent element $u \in G$.
	Then $M$ is connected.
	By definition, $M \subseteq H^\sat$.
	If $M \ne H^\sat$, then 
	by  repeating this process with $M$ (possibly several times), we eventually generate all of 
	$H^\sat$ by $H$ and closed connected subgroups of $G$
	isomorphic to ${\mathbb G}_a$. 
\end{rem}

Here is a further consequence of Theorem \ref{thm:saturation}.

\begin{cor}
	\label{cor:HcrvsGcr3}
	Let $p \ge \widetilde{h}(G)$. Let $K \subseteq H$ be closed subgroups of $G$ with $H$ connected reductive, and suppose that $p\geq \widetilde{h}(H)$.
	Then the following are equivalent:
	\begin{enumerate}[label=\rm(\roman*)]
		\item  $K$ is $H^\sat$-completely reducible;
		\item  $K^\sat$ is $H^\sat$-completely reducible;
		\item  $(K^\sat)^\circ$ is reductive;
		\item  $K^\sat$ is $G$-completely reducible;
		\item  $K$ is $G$-completely reducible. 
	\end{enumerate} 
\end{cor}

\begin{proof}
	Owing to Remark \ref{rem:Hsatconnected}, $H^\sat$ is connected. 
	Further, since 	$\widetilde{h}(G) \ge a(G)$, it follows from Theorems \ref{thm:ag} and \ref{thm:saturation} that 
	$H^\sat$ is reductive.
	
	The equivalence of (i) through (iii) follows from Theorem \ref{thm:saturation} applied to $K\subseteq H^\sat$ and Theorem \ref{thm:hGhH} and the equivalence of (iii) through (v) is just Theorem \ref{thm:saturation}.
\end{proof}

\begin{proof}[Proof of Theorem~\ref{thm:sat_cr}]
	Thanks to Remark \ref{rem:Hsatconnected}, $H^\sat$ is connected. 
	Since 	$d(G) \ge h(G) \ge a(G)$ by \eqref{eqn:inequalities}, it follows from Theorems \ref{thm:ag} and \ref{thm:saturation} that 
	$H^\sat$ is reductive.
	
	If $K$ is $H$-cr, then $K$ is $G$-cr, by Theorem~\ref{thm:main}, so   
	$K^\sat$ is $H^\sat$-cr, by Corollary \ref{cor:HcrvsGcr3}.		
\end{proof}

The following example illustrates that in general connected reductive subgroups are not saturated.

\begin{ex}
	\label{ex:SL_p}
	With the explicit notion of saturation from \eqref{eq:ut} within $\GL(V)$ it is easy to check that the image $H$ of the adjoint representation of $\SL_p$ in $G := \GL(\Lie(\SL_p))$ is not saturated in characteristic $p$, see \cite[p18]{serre1}.
	Evidently, $H$ is contained in the maximal parabolic subgroup $P$ of $G$ that stabilises the $H$-submodule $\zz(\Lie(\SL_p))$. One checks that its saturation $H^\sat$ in $G$ includes all of $H$ but also part of the 
	unipotent radical $R_u(P)$ of $P$.  For instance, when $p= 2$ then the adjoint representation of $H:= \SL_2$ with respect to a suitable basis is given by
	$$ {\rm Ad}\left(\left(\begin{array}{cc} a & b \\ c & d \end{array}\right)  \right)= \left(\begin{array}{ccc} a^2 & b^2 & 0 \\ c^2 & d^2 & 0 \\ ac & bd & 1 \end{array}\right). $$
	If $u= {\rm Ad}\left(\left(\begin{array}{cc} 1 & b \\ 0 & 1 \end{array}\right)\right)$ then ${\rm log}\,u= \left(\begin{array}{ccc} 0 & b^2 & 0 \\ 0 & 0 & 0 \\ 0 & b & 0 \end{array}\right)$, so $u^t= {\rm exp}(t\,{\rm log}\,u)= \left(\begin{array}{ccc} 1 & tb^2 & 0 \\ 0 & 1 & 0 \\ 0 & tb & 1 \end{array}\right)$ for any $t\in {\mathbb G}_a$.  We see that if $b\neq 0, 1$ then $u^{b^2}{\rm Ad}\left(\left(\begin{array}{cc} 1 & b^2 \\ 0 & 1 \end{array}\right)\right)=  \left(\begin{array}{ccc} 1 & 0 & 0 \\ 0 & 1 & 0 \\ 0 & b^2+ b^3 & 1 \end{array}\right)$ is a non-trivial element of $H^\sat\cap R_u(P)$, where $P$ is the parabolic subgroup of matrices of shape $\left(\begin{array}{ccc} * & * & 0 \\ * & * & 0 \\ * & * & * \end{array}\right)$.

		Since the abelian unipotent radical $R_u(P)$ is an irreducible $\SL_p$-module (of highest weight $\lambda_1 + \lambda_{p-1}$, or $2\lambda_1$ when $p = 2$), being a non-zero $\SL_p$-submodule of  $R_u(P)$, $U$ is in fact all of $R_u(P)$.
	So $H^\sat$ is of the form $H^\sat = XR_u(P)$, where $X$ is a subgroup of the Levi subgroup of type $\SL_{p^2 - 2}$ of $P$.  In particular, $H^\sat$ is not reductive in this case.
\end{ex}

We briefly revisit Example \ref{ex:notHsigmacr} in the context of saturation.

\begin{ex}
	\label{ex:notHsigmacr2}
	With the hypotheses and notation from Example \ref{ex:notHsigmacr}, a non-trivial unipotent element $u$ from $H_\sigma$ has order $p = 3$, so we can saturate $u$ in $H$ and in $G$, according to \eqref{eq:ut} above.
	Likewise we can saturate non-trivial unipotent elements in $H$. It turns out that $H$ is not saturated in $G$.
	Let $u = {\twobytwo{1}{b}{0}{1}}$ be in $\SL_2(9) = H_\sigma$ for a fixed $b \ne 1$.
	Then one can check that the saturations of $u$ in $H$ and in $G$ do not coincide. Thus, the hypotheses of Theorem \ref{thm:hGhH} fail on two accounts, for $H$ is not saturated in $G$ and $p = 3< h(G) = \dim V = 8$, while $p > h(H) = 2$.
\end{ex}

\section{Saturation and finite groups of Lie type}
\label{sec:sat-finite}

In this section we discuss finite subgroups of Lie type in $G$ and their behaviour under saturation.  To do this we need to prove the compatibility of the saturation map with standard Frobenius endomorphisms.  First recall that if $\sigma_q:G\to G$ is a standard $q$-power Frobenius endomorphism of $G$, then there exists a $\sigma_q$-stable
maximal torus $T$ and Borel subgroup $B\supseteq T$, and with respect to a chosen parametrisation of
the root groups as above, we have $\sigma_q(x_\alpha(s)) = x_\alpha(s^q)$ for each $\alpha\in \Phi$ and $s\in \BBG_a$, cf.~\cite[Thm.~1.15.4(a)]{GLS:1998}.

\begin{prop}
	\label{prop:frobsat}
	Let $p \ge \widetilde{h}(G)$.
Suppose $\sigma_q:G\to G$ is a standard $q$-power Frobenius endomorphism of $G$. Then the following hold:
\begin{enumerate}[label=\rm(\roman*)]
	\item  $\sigma_q(u^t) = \sigma_q(u)^{t^q}$ for any $u\in \UU$, $t \in \BBG_a$;
	\item if $H$ is a $\sigma_q$-stable subgroup of $G$, then $H^\sat$ is also $\sigma_q$-stable.
\end{enumerate}
\end{prop}

\begin{proof}
(i). Fix a $\sigma_q$-stable Borel subgroup $B$ of $G$ as in the discussion before the statement of the proposition, with unipotent radical $U$. 
Since we have $(gug^{-1})^t = gu^tg^{-1}$ for all $u\in \UU$, $g\in G$
thanks to Corollary \ref{cor:centralizers}(i),
it is enough to show the result for $u\in U$.

There are two ways to define a Frobenius-type map on $\Lie(U)$.
Firstly, since the $X_\alpha$ form a basis for $\Lie(U)$ as a $k$-space, we have the map $F_q:\Lie(U) \to \Lie(U)$
given by 
$$
\sum_{\alpha\in \Phi^+} c_\alpha X_\alpha \mapsto \sum_{\alpha\in\Phi^+} c_\alpha^qX_\alpha.
$$
For this map, it is clear that $F_q(tX) = t^qF_q(X)$ for every $t\in \BBG_a$ and $X\in \Lie(U)$.
Alternatively, since $\exp$ and $\log$ are mutually inverse group isomorphisms
between $U$ and $\Lie(U)$, there is some endomorphism $f_q:\Lie(U)\to \Lie(U)$
defined by 
$$
\sigma_q(\exp(X)) = \exp(f_q(X))
$$ 
for all $X\in \Lie(U)$ (or, equivalently, by $\log \sigma_q(u) = f_q(\log u)$ for all $u\in U$).
We claim that $f_q=F_q$.

First, note that Theorem \ref{thm:log}(iii) gives equality straight away for multiples of basis elements: 
since $\sigma_q(x_\alpha(s)) = x_\alpha(s^q)$ for each $s\in \BBG_a$ and positive root $\alpha$, we have
$f_q(sX_\alpha) = s^qX_\alpha = F_q(sX_\alpha)$ for all such $s$ and $\alpha$.
Now recall that if we fix some ordering of the positive roots $\Phi^+$, 
then each element $u\in U$ has a unique expression as a product $u = \prod_{\alpha\in \Phi^+} x_\alpha(s_\alpha)$ with the $s_\alpha\in \BBG_a$, cf.~\cite[Ch.~I, 1.2(b)]{springersteinberg}, and hence
every $X = \log(u)\in \Lie(U)$ has expression $X = \prod_{\alpha\in \Phi^+} (s_\alpha X_\alpha)$, where this product is calculated using the Baker-Campbell-Hausdorff formula, by Theorem \ref{thm:log}(ii) and (iii).
Thus, to show $f_q = F_q$ it is enough to show that $F_q$ is a group homomorphism.

Let $X = \sum_{\alpha\in \Phi^+} s_\alpha X_\alpha$ and $Y = \sum_{\alpha\in \Phi^+} t_\alpha X_\alpha$ be two elements of $\Lie(U)$.
In calculating $XY$ with the Baker-Campbell-Hausdorff formula we get a number of commutators involving the $s_\alpha X_\alpha$ and $t_\beta X_\beta$
for positive roots $\alpha, \beta$.
Since the Lie bracket is bilinear, we can pull all the coefficients $s_\alpha$ and $t_\beta$ out to the front of each commutator, and hence 
write $XY$ as a linear combination of commutators in the $X_\alpha$.
All such commutators of degree greater than $1$ can be rewritten in $\Lie(U)$ as a linear combination of the $X_\alpha$
by applying the commutator relations recursively to write any $[X_\beta,X_\gamma]$ in terms of the $X_\alpha$, and then
expanding out and repeating.
The coefficients appearing in the commutation relations lie in the finite base field $\mathbb{F}_p$, and hence are fixed under the $q$-power map.
Thus, for any commutator $C$ in the root elements $X_\alpha$, we may conclude that $F_q(C) = C$.
This is enough to conclude that $F_q(XY) = F_q(X)F_q(Y)$ for any $X,Y\in \Lie(U)$, as claimed.

We can now deduce that for any $t\in \BBG_a$ and any $X\in \Lie(U)$, we have $f_q(tX) = t^qf_q(X)$, and hence
for any $u\in U$ and any any $t\in \BBG_a$, 
$$
\sigma_q(u)^{t^q} = \exp(t^q \log \sigma_q(u)) = \exp (t^q f_q(\log u)) = \exp(f_q (t\log u)) = \sigma_q(\exp (t\log u)) = \sigma_q(u^t),
$$ 
which completes the proof of (i).

(ii). This follows quickly from (i), since if $H$ is $\sigma_q$-stable and $M$ is any saturated subgroup of $G$ containing $H$,
then $\sigma_q(M)$ is another saturated subgroup of $G$ containing $H$:
for, if $t\in \BBG_a$ and $u\in \sigma_q(M)$ is unipotent, then we may find $s\in \BBG_a$ with $s^q = t$ (since $k = \bar{k}$ is perfect)
and $v\in M$ which is unipotent such that $u = \sigma_q(v)$.
Then $u^t = \sigma_q(v)^{s^q} = \sigma_q(v^s) \in \sigma_q(M)$ also.
Hence $H^\sat$, which is the smallest saturated subgroup containing $H$, must also be $\sigma_q$-stable.
\end{proof}
 
Combining Corollary \ref{cor:centralizers} and Proposition \ref{prop:frobsat}, we obtain the following.

\begin{cor}
	\label{cor:steinbergsat}
		Let $p \ge \widetilde{h}(G)$.
	Suppose $\sigma :G\to G$ is a Steinberg endomorphism of $G$ such that $\sigma = \tau \sigma_q$, where $\tau \in \Aut(G)$ and $\sigma_q$ is a   
	standard $q$-power Frobenius endomorphism of $G$. Then the following hold:
	\begin{enumerate}[label=\rm(\roman*)]
		\item  $\sigma(u^t) = \sigma(u)^{t^q}$ for any $u\in \UU$, $t \in \BBG_a$;
		\item if $H$ is a $\sigma$-stable subgroup of $G$, then $H^\sat$ is also $\sigma$-stable.
	\end{enumerate}
\end{cor} 
	
\begin{proof}
	(i). By Corollary \ref{cor:centralizers}(i) and Proposition \ref{prop:frobsat}(i), we have for any $u\in \UU$, $t \in \BBG_a$,
	\[
	\sigma(u^t) = \tau(\sigma_q(u^t)) = \tau(\sigma_q(u)^{t^q}) =  \tau(\sigma_q(u))^{t^q}  = \sigma(u)^{t^q},
	\]
	as desired.
	
	Part (ii) follows from the arguments in the proofs of Corollary \ref{cor:centralizers}(ii) and Proposition \ref{prop:frobsat}(ii) along with part (i).
\end{proof}

We note that in general, a Steinberg endomorphism of a reductive group $G$ need
not be of the form given in Corollary \ref{cor:steinbergsat}, e.g., see Example \ref{ex:permut}.

We now apply Theorem \ref{thm:sat_cr} to the case when $K = H_\sigma$ for a Frobenius endomorphism $\sigma$ of $G$.  The next result is immediate from Lemma \ref{lem:GsigmaGcr} and Theorem \ref{thm:sat_cr}. 

\begin{cor}
	\label{cor:Hsigmasatcr1}
	Let $H \subseteq G$ be connected reductive groups. Suppose $p \ge d(G)$  and $p \ge \widetilde{h}(G)$.
	Let $\sigma\colon G\to G$ be a Steinberg endomorphism that stabilises $H$.  
	Then $(H_\sigma)^\sat$ is $H^\sat$-completely reducible. 
\end{cor}

We can potentially improve the bound on $p$ in the last corollary at the expense of imposing the conditions from 
Proposition \ref{prop:HcrvsHsigmacr}, as follows.

\begin{cor}
	\label{cor:Hsigmasatcr}
Suppose $G, H$ and $\sigma$ satisfy the hypotheses of Proposition \ref{prop:HcrvsHsigmacr}. Suppose in addition that $p\geq 
\widetilde{h}(G)$. 
	Then $(H_\sigma)^\sat$ is $H^\sat$-completely reducible. 
\end{cor}

\begin{proof}
	Since $p \ge \widetilde{h}(G) \ge a(G)$ by \eqref{eqn:inequalities}, $H$ is $G$-cr, by Theorem \ref{thm:ag}. Thus $H_\sigma$ is $G$-cr, by Proposition \ref{prop:HcrvsHsigmacr}. The result now follows from Corollary \ref{cor:HcrvsGcr3}. 
\end{proof}

Note that in 
Theorem \ref{thm:sat_cr} and 
Corollaries \ref{cor:Hsigmasatcr1} and \ref{cor:Hsigmasatcr} 
$H^\sat$ is again connected reductive. This follows from the fact that $\widetilde{h}(G) \ge a(G)$, Theorems \ref{thm:ag},  
\ref{thm:saturation} and Remark \ref{rem:Hsatconnected}.

Example \ref{ex:Tsigma} shows that generically the conditions of  
Corollary \ref{cor:Hsigmasatcr} are fulfilled.
Nevertheless, Example \ref{ex:notHsigmacr} and 
Corollary \ref{cor:HcrvsGcr3} show that Corollary \ref{cor:Hsigmasatcr} is false if condition (i) of Proposition \ref{prop:HcrvsHsigmacr} is not satisfied.  
In the settings of Corollaries \ref{cor:Hsigmasatcr1} and \ref{cor:Hsigmasatcr}, 
 $((H_\sigma)^\sat)^\circ$ is reductive.

Assume $G$ is simple for the rest of this section unless specified otherwise. Let $\sigma\colon G\to G$ be a Steinberg endomorphism.  Then $\sigma$ 
is a generalized Frobenius map, i.e., a suitable power of $\sigma$ is a standard Frobenius map 
(e.g., see \cite[Thm.\ 2.1.11]{GLS:1998}), 
and the possibilities for $\sigma$ are well known 
(\cite[\S 11]{steinberg:end}): 
$\sigma$ is conjugate to either $\sigma_q$, 
$\tau \sigma_q$, $\tau' \sigma_q$ or $\tau'$, 
where $\sigma_q$ is a standard Frobenius 
morphism, $\tau$ is an automorphism of 
algebraic groups coming from a graph 
automorphism of types $A_n$, $D_n$ or $E_6$, 
and $\tau'$ is a bijective endomorphism 
coming from a graph automorphism of type 
$B_2$ ($p=2$), $F_4$ ($p=2$) or $G_2$ ($p=3$). 
The latter instances only occur in bad characteristic, so are not relevant here.
If $\tau = 1$, then we say that $G_\sigma$ is \emph{untwisted}, else $G_\sigma$ is \emph{twisted}.
Note that, since $G$ is simple, $\tau$ and $\sigma_q$ commute.
Note also that $C_G(\tau)$ is  again simple (e.g., see \cite[Thm.\ 1.15.2(d)]{GLS:1998}).

\begin{thm}
	\label{thm:Hsigmasat=H}
	Suppose $G$ is simple.	
	Let $\sigma = \tau \sigma_q$ be a Steinberg endomorphism of $G$ and $H$ a connected semisimple $\sigma$-stable subgroup of $G$.  Assume $p \ge \widetilde{h}(G)$ and $p \ge \widetilde{h}(H)$. 
	Then $H^\sat$ is also $\sigma$-stable, and we have:
	\begin{enumerate}[label=\rm(\roman*)]
		\item  if $\tau = 1$, then $(G_\sigma)^\sat = G$;
		\item  if $\tau = 1$ and $H$ is saturated in $G$, then $(H_\sigma)^\sat = H$;
		\item  		$(G_\sigma)^\sat = C_G(\tau)$;
		\item  
		if $H$ is saturated in $G$, and both $\tau$ and $\sigma_q$ stabilise $H$ separately, then $(H_\sigma)^\sat = C_H(\tau)$;
		\item  if $H$ is saturated in $G$, then
		$((H_\sigma)^\sat)_\sigma = H_\sigma$.
	\end{enumerate}
\end{thm}

\begin{proof}
The fact that $H^\sat$ is $\sigma$-stable
follows from Corollary \ref{cor:steinbergsat}.

	For the rest of the proof, there is no loss in assuming that both $G$ and $H$ are  generated by their respective root subgroups relative to some fixed maximal $\sigma$-stable tori $T_H \subseteq T_G = T$.
	
	(i) and (ii). 
	If $\tau = 1$, i.e., if $\sigma = \sigma_q$ is standard, then every root subgroup of $G$ meets $G_\sigma$ non-trivially. (For, each root subgroup $U_\alpha$ of $G$ is $\sigma$-stable and the $\sigma$-stable maximal torus $T$ acts transitively on $U_\alpha$. So the result follows from the Lang-Steinberg Theorem.)
	It thus follows from Theorem \ref{thm:log}(iii) and \eqref{eq:sat} that $(G_\sigma)^\sat$ contains each root subgroup of $G$;
	thus (i) follows.
	The same argument applies for (ii) by considering the simple components of $H$ and the fact that saturation in $H$ coincides with saturation in $G$, by Theorem \ref{thm:hGhH}(ii).
	
	(iii).
	Since $\tau$ and $\sigma_q$ commute, we have
	$G_\sigma = C_G(\tau)_{\sigma_q}$.
	Since $C_G(\tau)$ is saturated in $G$, by Corollary \ref{cor:centralizers}(iii), the result follows from part (ii).
	
	(iv). Again, since $\tau$ and $\sigma_q$ commute, we have
	$H_\sigma = C_H(\tau)_{\sigma_q}$.
	Now $C_H(\tau)$ is saturated in $H$, by Corollary \ref{cor:centralizers}.  But since $H$ is saturated in $G$, saturation in $H$ coincides with saturation in $G$, by Theorem \ref{thm:hGhH}(ii), so the result follows from part (ii).
	
	(v). Thanks to Corollary \ref{cor:steinbergsat}, $(H_\sigma)^\sat$ is $\sigma$-stable. Thus, since $H_\sigma \subseteq (H_\sigma)^\sat$ and 
	$(H_\sigma)^\sat \subseteq H^\sat = H$,  
	we have $H_\sigma \subseteq ((H_\sigma)^\sat)_\sigma \subseteq H_\sigma$, and equality follows.	
\end{proof}

\begin{proof}[Proof of Theorem~\ref{thm:nori}]
 This follows immediately from Theorem~\ref{thm:Hsigmasat=H}(iii).
\end{proof}

\begin{rem}
	\label{rem:Hsigmasat=H}
	We note that Theorem \ref{thm:Hsigmasat=H}(v) generalises \cite[Thm.~B(1)]{nori}:
	If $G = \SL_n(k)$, $\sigma = \sigma_q$ is a standard Frobenius endomorphism of $G$, and $H$ is a  $\sigma$-stable subgroup of $G$, then it follows directly from \eqref{eq:ut} that $H^\sat$ is  again $\sigma$-stable. Thus in particular, if $H$ is a connected, saturated semisimple $\sigma$-stable subgroup of $G$, then by Theorem \ref{thm:Hsigmasat=H}(v) we have $((H_\sigma)^\sat)_\sigma = H_\sigma$;
	see \cite[Thm.~B(1)]{nori}.
\end{rem}

We consider some explicit examples for Theorem \ref{thm:Hsigmasat=H}.

\begin{ex}
	\label{ex:unitary}
	Let $G = \SL_n$ and let $\sigma$ be the Steinberg endomorphism of $G$  given by
	\[
	g \mapsto \sigma_q (n_0 ({}^tg^{-1}) n_0^{-1} ),
	\]
	where ${}^tg$ denotes the transpose of the matrix $g$ and $n_0$ is 
	the antidiagonal permutation matrix of $\GL_n$, normalizing $\SL_n$. Note that $\tau(g) = n_0 (^tg^{-1})n_0^{-1}$ is the graph automorphism of $G$.
	Then $\sigma^2 = \sigma_{q^2}$ is a standard  Frobenius map of $G$ given by raising coefficients
	to the $(q^2)^{\rm th}$ power.
	Note that $G_\sigma = \SU(q)$ is the special unitary subgroup of $G$.
	We have $\SU(q) = G_\sigma \subseteq G_{\sigma^2} = \SL(\BBF_{q^2})$, and since $\sigma_q$ commutes with $\tau$, we have (assuming $p \ge n$)
	$(G_\sigma)^\sat = C_G(\tau)$, by Theorem \ref{thm:Hsigmasat=H}(iii), while  $(G_{\sigma^2})^\sat = G$, by Theorem \ref{thm:Hsigmasat=H}(i).
\end{ex}

In the case when $G$ is no longer simple, additional kinds of Steinberg endomorphisms are possible.

\begin{ex}
	\label{ex:permut}
		Let $H$ be a semisimple group defined over $\BBF_q$ and let $\sigma_q$ be the corresponding standard Frobenius map of $H$. Let $G = H \times \cdots \times H$ ($r$ factors) and let $\Delta H$ be the diagonal copy of $H$ in $G$.
	Let $\pi$ be the $r$-cycle permuting the $r$ direct copies of $H$ of $G$ cyclically and let $f = (\sigma_q, id_H, \ldots, id_H) : G \to G$. Then $\sigma = \pi f$ is a Steinberg endomorphism of $G$ where $\pi$ and $f$ do not commute.
	We have $G_\sigma = (\Delta H)_{\sigma_{q}}$, where by abuse of notation $\sigma_{q}$ is a standard Frobenius map on $\Delta H$. (Note that $(\Delta H)_{\sigma_{q}}$ is isomorphic to $H_{\sigma_{q}} = H(\BBF_{q})$.)
	Now suppose $p\ge \widetilde{h}(H) = \widetilde{h}(G)$. Then $\Delta H$ is saturated in $G$, by Lemma~\ref{lem:products}.
	Thus $(G_\sigma)^\sat = ((\Delta H)_{\sigma_{q}})^\sat = \Delta H$, by Theorem \ref{thm:Hsigmasat=H}(ii).
\end{ex} 

We present an instance where Theorem \ref{thm:Hsigmasat=H}(i) can be applied even 
though $G$ is not simple and $\sigma$ is 
a Steinberg endomorphism which is not a generalized Frobenius endomorphism.

\begin{ex}
	\label{ex:sl2sl2}
	Let $p \ge 2$.
	Let $\sigma_p, \sigma_{p^2}$ be the standard Frobenius maps of $\SL_2$ given by raising coefficients
	to the $p^{\rm th}$ and $(p^2)^{\rm th}$ powers, respectively. Let $G = \SL_2 \times \SL_2$.
	Then the map 
	$\sigma = \sigma_p \times \sigma_{p^2} : G  \rightarrow G$
	is a Steinberg morphism of $G$ that is not a generalized Frobenius morphism (cf.\
	the remark following \cite[Thm.~2.1.11]{GLS:1998}).
	We have $G_\sigma = \SL_2(\BBF_p) \times \SL_2(\BBF_{p^2})$. 
	The saturation map in $G$ is given by the formula from \eqref{eq:ut}.  (If $p\geq 5$ then this follows from Theorem~\ref{thm:hGhH}, since the image of the canonical embedding of $G$ into $\GL_4(k)$ is a saturated subgroup of $\GL_4(k)$, but it is easily verified for $p< 5$ also.)  By Lemma~\ref{lem:products}, saturating $G_\sigma$ inside $G$ amounts to saturating each factor of $G_\sigma$ inside each factor of $G$. Now, by applying Theorem \ref{thm:Hsigmasat=H}(i) to each factor of $G$, we get $(G_\sigma)^\sat = G$.
\end{ex}

\section{Saturation and semisimplification}
\label{sec:semisat}

	\begin{defn}[{\cite[Def.~4.1]{BMR:semisimplification}}]
		Let $H$ be a subgroup of $G$. We say that a subgroup $H'$ of $G$ is a \emph{semisimplification of $H$ (for $G$)} if there exists a parabolic subgroup $P=P_\lambda$ of $G$ and a Levi subgroup $L=L_\lambda$ of $P$ such that $H\subset P_\lambda$ and $H'=c_\lambda(H)$, and $H'$ is $G$-completely reducible. We say the pair $(P,L)$ \emph{yields} $H'$.
	\end{defn}

The following consequence of Proposition \ref{prop:P_saturated} shows that passing to a semisimplification of a subgroup of $G$ and saturation are naturally compatible.

	\begin{cor}
		\label{cor:semisat}
		Suppose $p\geq \widetilde{h}(G)$.  Let $H'$ be a semisimplification of $H$ yielded by $(P,L)$. Then a semisimplification of $H^\sat$ is given by $\left(H'\right)^\sat$ and is yielded also by $(P,L)$. Moreover, any semisimplification of $H^\sat$ is $G$-conjugate to $\left(H'\right)^\sat$.
	\end{cor}

\begin{proof}
	By Lemma \ref{lem:Plambda} there exists a $\lambda \in Y(G)$ such that $P=P_\lambda, L=L_\lambda$ and $H'=c_\lambda(H)$. According to Proposition \ref{prop:P_saturated}, we have $\left(H'\right)^\sat=\left(c_\lambda(H)\right)^\sat=c_\lambda(H^\sat)$. 
	Since $H'$ is $G$-cr, by definition, $(H')^\sat$ is $G$-cr by Theorem \ref{thm:saturation}.	
	The final statement follows from \cite[Thm.~4.5]{BMR:semisimplification}.
	\end{proof}

In Example \ref{ex:SL_p}, the subgroup $H$ considered is connected, non-saturated and not $G$-cr. In our next example, we give a connected, non-saturated but $G$-cr subgroup.

\begin{ex}
	\label{ex:nonsatGcr}
	Consider the semisimple $\SL_2$-module $L(1)\oplus  L(p) = L(1)\oplus L(1)^{[p]}$, i.e., $\SL_2$ acting with a Frobenius twist on the second copy of the natural module and without such on the first copy. This defines a diagonal
	embedding of $\SL_2$ in $M := \SL_2 \times \SL_2 \subseteq G = \GL_4$. 
	The image $H$ of $\SL_2$ in $G$ is $G$-cr and it is not saturated in $G$ (and also not in the saturated $G$-cr subgroup $M$ of $G$). 
	The argument is similar to the one in Example \ref{ex:SL_p}.
	Note that $M$ is the saturation of $H$ in $G$. 
\end{ex}

We close this section by noting that in general homomorphisms are not compatible with saturation.
For instance, take the inclusion of a connected reductive, non-saturated subgroup $H$ in $G$.
See also Examples \ref{ex:SL_p} and \ref{ex:nonsatGcr}. 

\bigskip
\noindent {\bf Acknowledgments}: We are grateful to J-P.~Serre for some suggestions on an earlier version of the manuscript.
The research of this work was supported in part by
the DFG (Grant \#RO 1072/22-1 (project number: 498503969) to G.~R\"ohrle).  For the purpose of open access, the authors have applied a Creative Commons Attribution (CC BY) licence to any Author Accepted Manuscript version arising from this submission.

%\bigskip
%%%%%%%%%%%%%%%%%%%%%%%%%%%%%%%%%%%%%%%%%%%%%%%%%%%%%%%%%%%%%%%%%%%%%%
%%%%%%%%%%%%% bibliography
%%%%%%%%%%%%%%%%%%%%%%%%%%%%%%%%%%%%%%%%%%%%%%%%%%%%%%%%%%%%%%%%%%%%%%

\end{document}